\documentclass{zariski}

\title{A Foundation for Synthetic Stone Duality}

\begin{document}

\author{Felix Cherubini$^1$, Thierry Coquand$^1$, Freek Geerligs$^1$ and Hugo Moeneclaey$^1$}
\date{
  $^1$University of Gothenburg and Chalmers University of Technology \\[2ex]%
  \today
}

\maketitle
\begin{abstract}
  The language of homotopy type theory has proved to be appropriate as an internal language for various higher toposes, 
for example with Synthetic Algebraic Geometry for the Zariski topos.
In this paper we apply such techniques to the higher topos corresponding to the light condensed sets 
of Dustin Clausen and Peter Scholze.
This seems to be an appropriate setting to develop synthetic topology, similar to the work of 
Martín Escardó.
To reason internally about light condensed sets, we use homotopy type theory extended with 4 axioms.
Our axioms are strong enough to prove Markov's principle, LLPO and the negation of WLPO. 
We also define a type of open propositions, inducing a topology on any type. 
This leads to a synthetic topological study of (second countable)
Stone and compact Hausdorff spaces. 
Indeed all functions are continuous in the sense that they respect this induced topology, 
and this topology is as expected for these classes of types.
For example, any map from the unit interval to itself is continuous in the usual epsilon-delta sense.
We also use the synthetic homotopy theory 
given by the higher types of homotopy type theory to define and 
work with cohomology.
As an application, we prove Brouwer's fixed-point theorem
internally. 

\end{abstract}

\section*{Acknowledgements}
The idea to use the topological characterization of stone spaces as totally disconnected, compact Hausdorff spaces to prove \Cref{stone-sigma-closed} was explained to us by Martín Escardó.
We profited a lot from a discussion with Reid Barton and Johann Commelin. 
David Wärn noticed that Markov's principle (\Cref{MarkovPrinciple}) holds. 
At TYPES 2024, we had an interesting discussion with Bas Spitters on the topic of the article.
Work on this article was supported by the ForCUTT project, ERC advanced grant 101053291.

\section*{Introduction}
The language of homotopy type theory is a  dependent type theory enriched with the univalence axiom and higher inductive types. It has proven exceptionnally well-suited to
develop homotopy theory in a synthetic way \cite{hott}. It also provides
the precision needed to analyze categorical models of type theory \cite{vanderweide2024}.
Moreover, the arguments in this language can be rather directly represented in proof assistants. We use homotopy type theory to give a synthetic development of topology, which is analogous to the work on synthetic algebraic geometry \cite{draft}. 

We introduce 
four axioms which seem sufficient for expressing and proving basic notions of topology, based on the light condensed
sets approach, introduced in \cite{Scholze}.
Interestingly, this development establishes strong connections with constructive mathematics \cite{Bishop},
particularly constructive reverse mathematics \cite{ReverseMathsBishop,HannesDiener}. Several of Brouwer's principles, such that
any real function on the unit interval is continuous, or the celebrated fan theorem, are consequences of this system
of axioms. Furthermore, we can also prove principles that are not intuitionistically valid, such as Markov's Principle,
or even the so-called Lesser Limited Principle of Omniscience, a principle well studied in constructive reverse mathematics,
which is {\em not} valid effectively.

This development also aligns closely with the program of Synthetic
Topology \cite{SyntheticTopologyEscardo,SyntheticTopologyLesnik}:
there is a dominance of open propositions, providing any type with an intrinsic
topology, and we capture in this way synthetically the notion of (second-countable) compact Hausdorff spaces.
While working on this axiom system, we learnt about the related work \cite{bc24}, which provides a different axiomatisation
at the set level. We show that some of their axioms are consequences of our axiom system. In particular, we can introduce
in our setting a notion of ``Overtly Discrete'' spaces, dual in some way to the notion of compact Hausforff spaces, like
in Synthetic Topology\footnote{We actually have a
derivation of their ``directed univalence'', but this will be presented in a following paper.}.

A central theme of homotopy type theory is that the notion of {\em type} is more general than the notion of {\em set}. We illustrate
this theme here as well: we can form in our setting the types of Stone spaces and of compact Hausdorff spaces
(types which don't form a set but a groupoid),
and show these types are
closed under sigma types. It would be impossible to formulate such properties in the setting of simple type or set theory.
Additionally, leveraging the elegant definition of cohomology groups in homotopy type theory \cite{hott}, which relies
on higher types that are not sets, we prove, in a purely axiomatic way,
a special case of a theorem of Dyckhoff \cite{dyckhoff76}, describing
the cohomology of compact Hausdorff spaces. This characterisation also supports a type-theoretic proof of
Brouwer's fixed point theorem, similar to the proof in \cite{shulman-Brouwer-fixed-point}. In our setting the theorem can be formulated in the usual
way, and not in an approximated form.

It is important to stress that what we capture in this axiomatic way are the properties of light condensed
sets that are {\em internally} valid. David W\"arn \cite{warn2024} has proved that an important property of abelian
groups in the setting of light condensed sets, is {\em not} valid internally and thus cannot be proved in this axiomatic context.
We believe however that our axiom system can be convenient for proving the results that are internally valid, as we hope
is illustrated by the present paper. We also conjecture that the present axiom system is actually {\em complete}
for the properties that are internally valid. Finally, we think that this system can be justified in a constructive metatheory
using the work \cite{CRS21}.

\section{Stone duality}
\subsection{Preliminaries}
%
\begin{remark}
  For $X$ any type, a subtype $U$ is a family of propositions over $X$. 
  We write $U\subseteq X$.
  If $X$ is a set, we call $U$ a subset. Given $x:X$ we sometimes write $x\in U$ instead of $U(x)$. 
  For subtypes $A,B\subseteq X$, we write $A\subseteq B$ for pointwise implication.
  We will freely switch between subtypes $U\subseteq X$ 
  and the corresponding embeddings
  $
    \sum_{x:X}U(x) \hookrightarrow  X.
  $
In particular, if we write $x: U$ we mean $x:X$ such that $U(x)$.
\end{remark}
\begin{definition}
  A type is countable if and only if it is 
  merely equal to some decidable subset of $\N$. 
\end{definition}

\begin{definition}
  For $I$ a set we write $2[I]$ for the free Boolean algebra on $I$.
  A Boolean algebra $B$ is countably presented,
  if there exist countable sets $I,J$, 
  generators $g_i:C,~{i\in I}$ and relations $f_j:2[I],~{j\in J}$ 
  such that $g$ induces an equivalence between $B$ and $2[I]/(f_j)_{j:J}$.
\end{definition} 

\begin{remark}\label{BooleAsCQuotient}
Any countably presented algebra is merely of the form 
$2[\N]/ (r_n)_{n:\N}$.
\end{remark}


\begin{remark}
  We denote the type of countably presented Boolean algebras by $\Boole$. 
  This type does not depend on a choice of universe. 
  Moreover $\Boole$ has a natural category structure. 
\end{remark}

\begin{example}
  If both the set of generators and relations are empty, we have the Boolean algebra $2$.
  Its underlying set is $\{0,1\}$ and $0\neq_2 1$.
  $2$ is initial in $\Boole$. 
\end{example}
%
\begin{definition}
  For $B$ a countably presented Boolean algebra, 
  we define $Sp(B)$ as the set of Boolean morphisms from $B$ to $2$.
  Any type which is merely equivalent to a type of the form $Sp(B)$ is called a Stone space.
\end{definition}

\begin{example}
  \label{boolean-algebra-examples}
  \item 
  \begin{enumerate}[(i)]
  \item There is only one Boolean morphism from $2$ to $2$, thus $Sp(2)$ is the singleton type $\top$. 
  \item   
    The tivial Boolean algebra is given by $2/(1)$. 
    We have $0=1$ in the trivial Boolean algebra, so  
    there cannot be a map from it into $2$ preserving both $0$ and $1$.
    So the corresponding Stone space is the empty type $\bot$.
  \item\label{ExampleBAunderCantor}   
    The type $Sp(2[\N])$ is called the Cantor space. It is equivalent to the set of binary sequences $2^\N$. 
    If $\alpha:Sp(2[\N])$ and $n:\N$ we write $\alpha_n $ for $\alpha(g_n)$. 
  \item\label{ExampleBAunderNinfty}
    We denote by $B_\infty$ the Boolean algebra generated by 
    $(g_n)_{n:\N}$ quotiented by the relations $g_m \wedge g_n = 0$ for ${n\neq m}$.
    A morphism $B_\infty\to 2$ corresponds to a function 
    $\mathbb N \to 2$ that hits $1$ at most once. 
    We denote $Sp(B_\infty)$ by $\Noo$. 
    For $\alpha:\Noo$ and $n:\N$ we write $\alpha_n$ for $\alpha(g_n)$. 
  By conjunctive normal form, 
  any element of $B_\infty$ can be written uniquely as 
  $\bigvee_{i:I} g_n$ or as $\bigwedge_{i:I} \neg g_n$ for some finite $I\subseteq \N$. 
  \end{enumerate}
\end{example}

\begin{lemma}\label{ClosedPropAsSpectrum}
  For $\alpha:2^\N$, we have an equivalence of propositions: 
 \[ 
    (\forall_{n:\N} \alpha_n = 0 )\leftrightarrow Sp(2/(\alpha_n)_{n:\N}).
  \] 
\end{lemma}
\begin{proof}
  There is only one boolean morphism $x:2\to 2$, and it satisfies 
  $x(\alpha_n) = 0$ for all $n:\N$ if and only if
  $\alpha_n = 0$ for all $n:\N$. 
\end{proof}


\subsection{Axioms}\label{Axioms}
\begin{axiom}[Stone duality]\label{AxStoneDuality}
  For any $B:\Boole$, 
  the evaluation map $B\rightarrow  2^{Sp(B)}$ is an isomorphism.
\end{axiom} 


\begin{axiom}[Surjections are formal surjections]\label{SurjectionsAreFormalSurjections}
  For $g:B\to C$ a map in $\Boole$, $g$ is injective if and only if
  $(-)\circ g: Sp(C) \to Sp(B)$ is surjective. 
\end{axiom} 
\begin{axiom}[Local choice]\label{AxLocalChoice}
  Whenever we have $B:\Boole$, and some type family $P$ over $Sp(B)$ with 
  $\Pi_{s:Sp(B)} \propTrunc{P(s)}$, then there 
  merely exists some $C:\Boole$ and surjection $q:Sp(C)\to Sp(B)$ with 
$  \Pi_{t:Sp(C)} P(q(t))$.
\end{axiom}

\begin{axiom}[Dependent choice]\label{axDependentChoice}
Given types $(E_n)_{n:\N}$ with for all $n:\N$ a surjection $E_{n+1}\twoheadrightarrow E_n$, the projection from the sequential limit $\lim_kE_k$ to $E_0$ is surjective.
\end{axiom}

\subsection{Anti-equivalence of $\Boole$ and $\Stone$}
By \Cref{AxStoneDuality}, $Sp$ is an embedding of $\Boole$ into any universe of types. 
We denote its image by $\Stone$. 

\begin{remark}\label{SpIsAntiEquivalence}
Stone spaces will take over the role of affine scheme from \cite{draft}, 
and we repeat some results here. 
Analogously to Lemma 3.1.2 of \cite{draft}, 
for $X$ Stone, Stone duality tells us that $X = Sp(2^X)$. 
Proposition 2.2.1 of \cite{draft} now says that 
$Sp$ gives a natural equivalence 
\[
   Hom_{\Boole} (A, B) = (Sp(B) \to Sp(A))
\]
%
$\Stone$ also has a natural category structure.
By the above and Lemma 9.4.5 of \cite{hott}, 
the map $Sp$ defines a dual equivalence of categories between $\Boole$ and $\Stone$.
In particular the spectrum of any colimit in $\Boole$ is the limit of 
the spectrum of the opposite diagram. 
\end{remark}
\begin{remark}\label{LocalChoiceSurjectionForm}
  \Cref{AxLocalChoice} can also be formulated as follows:
  whenever we have $S:\Stone$, $E,F$ arbitrary types, a map $f:S \to F$ and a 
  surjection $e:E \twoheadrightarrow F$, 
  there exists a Stone space $T$, a surjective map 
  $T\twoheadrightarrow S$ and an arrow $T\to E$ making the following diagram commute:
    \[\begin{tikzcd}
      T \arrow[d,dashed, two heads ] \arrow[r,dashed]&  E \arrow[d,""',two heads, "e"]\\
      S  \arrow[r, swap,"f"] & F
    \end{tikzcd}\]  
\end{remark}

\begin{lemma}\label{SpectrumEmptyIff01Equal}
  For $B:\Boole$, we have $0=_B1$ if and only if $\neg Sp(B)$.
\end{lemma}
\begin{proof}
  If $0=_B1$, there is no map $B\to 2$ preserving both $0$ and $1$, thus $\neg Sp(B)$. 
  Conversely, if $\neg Sp(B)$, then 
  $Sp(B)$ equals $\bot$, the spectrum of the trivial Boolean algebra. 
  As $Sp$ is an embedding, $B$ is equivalent to the trivial Boolean algebra, hence $0=_B1$. 
\end{proof}

\begin{corollary}\label{LemSurjectionsFormalToCompleteness}
 For $S:\Stone$, we have that $\neg \neg S \to  \propTrunc{S}$
\end{corollary}
\begin{proof}
  Let $B:\Boole$ and suppose $\neg \neg Sp(B)$. By \Cref{SpectrumEmptyIff01Equal} we have that $0\not=_B1$, therefore the morphism $2\to B$ is injective. By \Cref{SurjectionsAreFormalSurjections} the map $Sp(B) \to Sp(2)$ is surjective, thus $Sp(B)$ is merely inhabited. 
\end{proof}

\subsection{Principles of omniscience}
In constructive mathematics, we do not assume the law of excluded middle (LEM).
There are some principles called principles of omniscience that are weaker than LEM, which can be used to describe 
how close a logical system is to satisfying LEM.
References on these principles include \cite{HannesDiener, ReverseMathsBishop}.
In this section, we will show that two of them (MP and LLPO) hold, 
and one (WLPO) fails in our system.

\begin{theorem}[The negation of the weak lesser principle of omniscience ($\neg$WLPO)]\label{NotWLPO}
  \[
    \neg \forall_{\alpha:2^\N} 
    ((\forall_{n:\N} \alpha_n = 0 ) \vee \neg (\forall_{n:\N} \alpha_n = 0))
  \]
\end{theorem}
\begin{proof}
  Assume $f:2^\mathbb N \to 2$ such that 
  $f(\alpha) = 0$ if and only if $\forall_{n:\mathbb N} \alpha_n= 0$. 
  By \Cref{AxStoneDuality}, there is some $c:2[\N]$ with 
  $f(\alpha) = 0 \leftrightarrow \alpha(c) = 0$. 
  There exists $k:\N$ such that $c$ is expressed the generators $(g_n)_{n\leq k}$. 
  Now consider $\beta,\gamma:2^\N$ given by 
  $\beta(g_n) = 0$ for all $n:\mathbb N$ and
  $\gamma(g_n) = 0$ if and only if $n\leq k$. 
  As $\beta, \gamma$ are equal on $(g_n)_{n\leq k}$, we have $\beta(c) = \gamma(c)$. 
  However, $f(\beta) = 0$ and $f(\gamma) = 1$, giving a contradiction. 
\end{proof}

\begin{theorem}
  For $\alpha:\Noo$, we have that 
  \[
    (\neg (\forall_{n:\mathbb N} \alpha_n= 0)) \to \Sigma_{n:\mathbb N} \alpha_n= 1
  \]
\end{theorem}
\begin{proof}
  By \Cref{ClosedPropAsSpectrum}, we have that $\neg(\forall_{n:\N} \alpha_n = 0)$ implies that 
  $Sp(2/(\alpha_n)_{n:\N}$ is empty. 
  Hence $2/(\alpha_n)_{n:\N}$ is trivial by \Cref{SpectrumEmptyIff01Equal}. 
  Then there exists $k:\N$ such that $\bigvee_{i\leq k} \alpha_i = 1$. 
  As $\alpha_i = 1$ for at most one $i:\N$, 
  there exists an unique $n:\mathbb N$ with $\alpha_n = 1$. 
\end{proof}

\begin{corollary}[Markov's principle (MP)]\label{MarkovPrinciple}
  For $\alpha:2^\mathbb N$, we have that 
  \[
    (\neg (\forall_{n:\mathbb N} \alpha_n= 0)) \to \Sigma_{n:\mathbb N} \alpha_n= 1
  \]
\end{corollary}
\begin{proof}
  Given $\alpha:2^\mathbb N$, consider the sequence $\alpha':\Noo$ satisfying $\alpha'_n = 1$ if and only if
  $n$ is minimal with $\alpha_n = 1$. Then apply the above theorem.
\end{proof}

\begin{theorem}[The lesser limited principle of omniscience (LLPO)]\label{LLPO}
  For $\alpha:\N_\infty$, 
  we have: 
  \[\label{eqnLLPO}
    \forall_{k:\N} \alpha_{2k} = 0  \vee \forall_{k:\N} \alpha_{2k+1} = 0
  \]
\end{theorem}
\begin{proof}
%
  Define $f:B_\infty \to B_\infty \times B_\infty$ on generators as follows:
  \[\label{eqnLLPOProofMap}
    f(g_n) =\begin{cases}
      (g_k,0) \text{ if } n = 2k\\
      (0,g_k) \text{ if } n = 2k+1\\
    \end{cases}
  \]
  Note that $f$ is well-defined as map in $\Boole$ as 
  $f(g_n) \wedge f(g_m) = 0$ whenever $m\neq n$. 
  We claim that $f$ is injective. 
  If $I\subseteq \N$, write 
  $ I_0 =\{k\ |\ 2k \in I\}, 
    I_1 =\{k\ |\ 2k+1 \in I\}
  $.
  Recall that any $x:B_\infty$ is of the form 
  $\bigvee_{i\in I} g_i$ or $\bigwedge_{i\in I} \neg g_i$ for some finite set $I$. 
  \begin{itemize}
    \item If $x = \bigvee_{i\in I} g_i$, then 
      $f(x) = (\bigvee_{i\in I_0}g_{i}, \bigvee_{i\in I_1}g_i)$. 
      So if $f(x) = 0$, then $I_0=I_1 = I = \emptyset$ and $x = 0$. 
    \item Suppose 
      $x = \bigwedge_{i\in I} \neg g_i$.
      Then $f(x) = (\bigwedge_{i\in I_0} \neg g_i, \bigwedge_{i\in I_1} \neg g_i)$, 
      so $f(x) \neq 0$. 
  \end{itemize}
  By \Cref{SurjectionsAreFormalSurjections},
  $f$ corresponds to a surjection 
  $s:\Noo + \Noo \to \Noo$.
  Thus for $\alpha : \Noo$, 
  there exists some $x:\Noo + \Noo$ such that $s(x) = \alpha$. 
  If $x = inl(\beta)$, 
  for any $k:\N$, we have that 
\[\alpha_{2k+1} = s(x)_{2k+1} = x(f(g_{2k+1})) = inl(\beta) (0,g_k)  = \beta(0) = 0.\]
  Similarly, if $x = inr(\beta)$, we have $\alpha_{2k} = 0$ for all $k:\N$. 
\end{proof}
The surjection $s:\Noo + \Noo \to \Noo$ as above does not have a section 
as the following shows:
\begin{lemma}
  The function $f$ defined above does not have a retraction. 
\end{lemma}
\begin{proof}
  Suppose $r:B_\infty \times B_\infty \to B_\infty$ is a retraction of $f$. 
  Then $r(0,g_k) = g_{2k+1}$ and $r(g_k,0) = g_{2k}$. 
  Note that $r(0,1) \geq r(0,g_k) = g_{2k+1}$ for all $k:\N$. 
  As a consequence, $r(0,1)$ is of the form $\bigwedge_{i\in I} \neg g_i$ for some finite set $I$.
  By similar reasoning so is $r(1,0)$. 
  But this contradicts:
  \[r(0,1) \wedge r(1,0) = r( (1,0) \wedge (0,1)) = r(0,0) = 0.\]
  Thus no retraction exists. 
\end{proof}

\subsection{Open and closed propositions}
In this section we will introduce a topology on the type of propositions, and 
study their logical properties.
We think of open and closed propositions respectively as countable disjunctions and conjunctions of decidable propositions.
Such a definition is universe-independent, and can be made internally.

\begin{definition}
A proposition $P$ is open (resp. closed) if there exists some $\alpha:2^\N$ such that $P \leftrightarrow \exists_{n:\mathbb N} \alpha_n = 0$ (resp. $P \leftrightarrow \forall_{n:\mathbb N} \alpha_n = 0$). We denote by $\Open$ and $\Closed$ the types of open and closed propositions.
\end{definition}

\begin{remark}\label{rmkOpenClosedNegation}
  The negation of an open proposition is closed, 
  and by MP (\Cref{MarkovPrinciple}), the negation of a closed proposition is open 
  and both open, closed propositions are $\neg\neg$-stable. 
  By $\neg$WLPO (\Cref{NotWLPO}), 
  not every closed proposition is decidable. 
  Therefore, not every open proposition is decidable. 
  Every decidable proposition is both open and closed.
\end{remark}
\begin{lemma}
  We have the following results on open and closed propositions:
  \begin{itemize}
    \item Closed propositions are closed under finite disjunctions. 
    \item Closed propositions are closed under countable conjunctions. 
    \item Open propositions are closed under finite conjunctions. 
    \item Open propositions are closed under countable disjunctions. 
  \end{itemize}
\end{lemma}
\begin{proof}
  By Proposition 1.4.1 of \cite{HannesDiener}, LLPO(\Cref{LLPO}) is equivalent to the statement that 
  the disjunction of two closed propositions are closed. 
  The other statements have similar proofs, and we only present the proof that closed propositions are closed under 
  countable conjunctions. 
  Let $(P_n)_{n:\N}$ be a countable family of closed propositions. 
  By countable choice, for each 
  $n:\N$ we have an $\alpha_n:2^\N $ 
  such that $P_n \leftrightarrow \forall_{m:\N} \alpha_{n,m} =0$. 
  Consider a surjection $s:\N \twoheadrightarrow \N \times \N$, and let 
  $\beta_k = \alpha_{s(k)}.$
  Note that $\forall_{k:\N} \beta_k = 0$ if and only if 
  $\forall_{n:\N} P_n$. 
\end{proof}
We will use the above properties silently from now on. 
\begin{corollary}\label{ClopenDecidable}
  If a proposition is both open and closed, it is decidable. 
\end{corollary}
\begin{proof}
  If $P$ is open and closed, 
  $P\vee \neg P$ is open, 
  hence $\neg\neg$-stable and provable. 
%
\end{proof}


\begin{lemma}\label{ClosedMarkov}
  For $(P_n)_{n:\N}$ a sequence of closed propositions, we have 
  $\neg \forall_{n:\N} P_n \leftrightarrow  \exists_{n:\N} \neg P_n$. 
\end{lemma}
\begin{proof}
  Both $\neg \forall_{n:\N} P_n$ and $\exists_{n:\N} \neg P_n$ are open, hence $\neg\neg$-stable.
  The equivalence follows. 
\end{proof} 

%
%
\begin{lemma}\label{ImplicationOpenClosed}
  If $P$ is open 
  and $Q$ is closed 
  then $P\to Q$ is closed. 
  If $P$ is closed and $Q$ open, then $P\to Q$ is open. 
\end{lemma}
\begin{proof}
  Note that $\neg P \vee Q$ is closed. Using $\neg\neg$-stability
  we can show $(P\to Q) \leftrightarrow (\neg P \vee Q)$. 
  The other proof is similar. 
\end{proof}
%

\subsection{Types as spaces}
The subobject $\Open$ of the type of propositions induces a topology on every type. 
This is the viewpoint taken in synthetic topology. 
We will follow the terminology of \cite{SyntheticTopologyEscardo, SyntheticTopologyLesnik}. 

\begin{definition}
  Let $T$ be a type, and let $A\subseteq T$ be a subtype. 
  We call $A\subseteq T$ open (resp. closed) if $A(t)$ is open (resp. closed) for all $t:T$.
\end{definition}

\begin{remark}
  It follows immediately that the pre-image of an open by any map of types is open, so that any map is continuous. 
  In \Cref{StoneClosedSubsets}, we shall see that the resulting topology is as expected for second countable Stone spaces.
  In \Cref{IntervalTopologyStandard}, we shall see that the same holds for the unit interval. 
\end{remark}


\section{Overtly discrete spaces}
%

\begin{definition}
  We call a type overtly discrete if
  it is a sequential colimit of finite sets. 
\end{definition} 
\begin{remark}
  It follows from Corollary 7.7 of \cite{SequentialColimitHoTT} that 
  overtly discrete types are sets, and that the colimit can be defined as in set theory. 
  The type of overtly discrete types is independent on a choice of universe, 
  so we can write $\ODisc$ for this type. 
\end{remark}
  Using dependent choice, we have the following results: 
  \begin{lemma}\label{lemDecompositionOfColimitMorphisms}
      A map between overtly discrete sets is a sequential colimit of maps between finite sets. 
  \end{lemma}
  \begin{lemma}\label{lemDecompositionOfEpiMonoFactorization}
    For 
    $f:A\to B$ a sequential colimit of maps of finite sets $f_n:A_n \to B_n$, we have 
      that the factorisation $A \twoheadrightarrow Im(f) \hookrightarrow B$ is the 
      sequential colimit of the factorisations 
      $A_n \twoheadrightarrow Im(f_n) \hookrightarrow B_n$. 
    \end{lemma}

\begin{corollary}\label{decompositionInjectionSurjectionOfOdisc}
  An injective (resp. surjective) map between overtly discrete types 
  is a sequential colimit of injective (resp. surjective) maps between finite sets. 
\end{corollary}

\subsection{Closure properties of $\ODisc$}


We can get the following result using \Cref{lemDecompositionOfColimitMorphisms} and dependent choice.

\begin{lemma}\label{ODiscClosedUnderSequentialColimits}
  Overtly discrete types are closed under sequential colimits. 
\end{lemma}

We have that $\Sigma$-types, identity types and propositional truncation commutes with sequential colimits (Theorem 5.1, Theorem 7.4 and Corollary 7.7 in \cite{SequentialColimitHoTT}). Then by closure of finite sets under these constructors, we can get the following: 

\begin{lemma}\label{OdiscSigma}
  Overtly discrete types are closed under $\Sigma$-type, identity type and propositional truncation.
\end{lemma}

\subsection{$\Open$ and $\ODisc$} 
\begin{lemma}\label{PropOpenIffOdisc}
  A proposition is open if and only if it is overtly discrete.
\end{lemma}
\begin{proof}
  If $P$ is overtly discrete, then $P\leftrightarrow \exists_{n:\N} \propTrunc{F_n}$ with $F_n$ finite sets. 
  But a finite set being merely inhabited is decidable, hence $P$ is a countable disjunction of decidable propositions, hence open.
  Suppose $P\leftrightarrow \exists_{n:\N} \alpha_n = 1$. 
  Let $P_n = \exists_{n\leq k} (\alpha_n = 1)$, which is a decidable proposition, hence a finite set. 
  Then the colimit of $P_n$ is $P$. 
\end{proof}

\begin{corollary}\label{OpenDependentSums}
  Open propositions are closed under sigma types. 
\end{corollary}
\begin{corollary}[transitivity of openness]\label{OpenTransitive}
  Let $T$ be a type, let $V\subseteq T$ open and let $W\subseteq V$ open. 
  Then $W\subseteq T$ is open as well. 
\end{corollary}

\begin{remark}\label{OpenDominance}
  It follows from  Proposition 2.25 of \cite{SyntheticTopologyLesnik} that 
  $\Open$ is a dominance in the setting of synthetic topology. 
\end{remark}

\begin{lemma}\label{OdiscQuotientCountableByOpen}\label{ODiscEqualityOpen}
  A type $B$ is overtly discrete if and only if it merely is the quotient of a countable set by an open equivalence relation. 
\end{lemma}
\begin{proof}
  If $B:\ODisc$ is the sequential colimit of finite sets $B_n$, 
  then $B$ is an open quotient of $ (\Sigma_{n:\N} B_n)$.
%
  Conversely, assume $B= D/R$ with $D\subseteq \N$ decidable and $R$ open. 
  By dependent choice we get $\alpha:D \to D \to 2^\N$ such that 
  $R(x,y)\leftrightarrow \exists_{k:\N}\alpha_{x,y}(k) = 1$. 
  Define $D_n = (D \cap \N_{\leq n})$, and $R_n : D_n \to D_n \to 2$ so that 
  $R_n(x,y)$ is the equivalence relation generated by the relation 
  $\exists_{k\leq n} \alpha_{x,y}(k) =1$. 
  Then the $B_n = D_n/R_n$ are finite sets, and have colimit $B$. 
\end{proof}

\subsection{Relating $\ODisc$ and $\Boole$}
\begin{lemma}\label{BooleIsODisc}
  Every countably presented Boolean algebra is merely a sequential colimit of finite Boolean algebras. 
\end{lemma}
\begin{proof}
  Consider a countably presented Boolean algebra of the form $B = 2[\N]/(r_n)_{n:\N}$. 
  For each $n:\N$, let $G_n$ be the union of $\{g_i\ |\ {i\leq n}\}$ and 
  the finite set of generators occurring in $r_i$ for some $i\leq n$. 
  Denote $B_n = 2[G_n]/(r_i)_{i\leq n}$. 
  Each $B_n$ is a finite Boolean algebra, and there are canonical maps $B_n \to B_{n+1}$.
  Then $B$ is the colimit of this sequence. 
%
%
\end{proof}

\begin{corollary}\label{ODiscBAareBoole}
  A Boolean algebra $B$ is overtly discrete if and only if it is countably presented. 
\end{corollary}
\begin{proof}
  Assume $B:\ODisc$. 
  By \Cref{OdiscQuotientCountableByOpen}, we get a surjection $\N\twoheadrightarrow B$ and that $B$ has open equality. 
  Consider the induced surjective morphism $f:2[\N]\twoheadrightarrow B$.
  By countable choice, we get for each $b:2[\N]$
  a sequence $\alpha_{b}:2^\N$ such that 
  $(f(b) = 0)\leftrightarrow \exists_{k:\N} (\alpha_{b}(k) =1)$. 
  Consider 
  $r:2[\N] \to \N \to 2[\N]$ 
  given by 
  \[r(b,k) =\begin{cases}
    b &\text{ if } \alpha_{b}(k) = 1\\
    0   &\text{ if } \alpha_{b}(k) = 0
  \end{cases}
  \] 
  Then $B= 2[\N]/(r(b,k))_{b:2^\N,k:\N}$.
  \Cref{BooleIsODisc} gives the converse.
\end{proof}

\begin{remark}\label{BooleEpiMono}
  By \Cref{OdiscSigma} and \Cref{ODiscBAareBoole}, 
  it follows that any 
  $g:B\to C$ in $\Boole$ has an overtly discrete kernel.
  As a consequence, the kernel is enumerable and $B/Ker(g)$ is in $\Boole$. 
  By uniqueness of epi-mono factorizations and \Cref{SurjectionsAreFormalSurjections}, 
  the factorization 
  $B\twoheadrightarrow B/Ker(g) \hookrightarrow C$ corresponds to 
  $Sp(C) \twoheadrightarrow Sp(B/Ker(g)) \hookrightarrow Sp(B)$. 
\end{remark}
\begin{remark}\label{decompositionBooleMaps}
  Similarly to \Cref{lemDecompositionOfColimitMorphisms} and 
  \Cref{lemDecompositionOfEpiMonoFactorization}  a map (resp. surjection, injection) 
  in $\Boole$ is a sequential colimit of maps (resp. surjections, injections) between 
  finite Boolean algebras. 
\end{remark}

\section{Stone spaces}

\subsection{Stone spaces as profinite sets}
Here we present Stone spaces as sequential limits of finite sets. 
This is the perspective taken in Condensed Mathematics \cite{Condensed,Dagur,Scholze}.
Some of the results in this section are specific versions of the axioms used in 
\cite{bc24}. A full generalization is part of future work. 

\begin{lemma}
  Any $S:\Stone$ is merely a sequential limit of finite sets. 
\end{lemma}
\begin{proof}
  Assume $B:\Boole$. By \Cref{SpIsAntiEquivalence} and \Cref{BooleIsODisc}, 
 we have that $Sp(B)$ is the sequential limit of the $Sp(B_n)$, which are finite sets. 
\end{proof}

\begin{lemma}\label{StoneAreProfinite}
  A sequential limit of finite sets is a Stone space. 
\end{lemma}
\begin{proof}
  By \Cref{SpIsAntiEquivalence} and 
  \Cref{ODiscClosedUnderSequentialColimits}, 
  we have that $\Stone$ is closed under sequential limits, and finite sets are Stone.
\end{proof}

\begin{corollary}
Stone spaces are stable under finite limits.
\end{corollary}
\begin{remark}\label{StoneClosedUnderPullback}\label{ProFiniteMapsFactorization}
  By \Cref{decompositionBooleMaps} and 
  \Cref{SurjectionsAreFormalSurjections}, maps (resp. surjections, injections) of Stone spaces
  are sequential limits of maps (resp. surjections, injections) of finite sets. 
%
%
%
%
\end{remark}

\begin{lemma}\label{ScottFiniteCodomain}
  For $(S_n)_{n:\N}$ a sequence of finite types with $S=\lim_nS_n$ and $k:\N$, we have that $\mathrm{Fin}(k)^{S}$ is the sequential colimit of $\mathrm{Fin}(k)^{S_n}$.
\end{lemma}
\begin{proof}
  By \Cref{SpIsAntiEquivalence} we have $\mathrm{Fin}(k)^S = \Hom(2^{k},2^S)$.
  Since $2^{k}$ is finite, we have that $\Hom(2^k,\_)$ commutes with sequential colimits, therefore $\Hom(2^{k},2^S)$ is the colimit of $\Hom(2^{k},2^{S_n})$. 
  By applying \Cref{SpIsAntiEquivalence} again, 
  the latter type is $\mathrm{Fin}(k)^{S_n}$.
\end{proof}

\begin{lemma}\label{MapsStoneToNareBounded}
  For $S:\Stone$ and $f:S \to \N$, there exists some $k:\N$ such that $f$ factors through $\mathrm{Fin}(k)$. 
\end{lemma}
\begin{proof}
  For each $n:\N$, the fiber of $f$ over $n$ is a decidable subset $f_n:S \to 2$. 
  We must have that $Sp(2^S/(f_n)_{n:\N}) = \bot$, hence there exists some $k:\N$ with 
  $\bigvee_{n\leq k} f_n =_{2^S} 1 $. 
  It follows that $f(s)\leq k$ for all $s:S$ as required. 
\end{proof}

\begin{corollary}\label{scott-continuity}
  For $(S_n)_{n:\N}$ a sequence of finite types with $S=\lim_nS_n$, we have that $\N^S$ is the sequential colimit of $\N^{S_n}$. 
\end{corollary}
\begin{proof}
  By \Cref{MapsStoneToNareBounded} we have that $\N^S$ is the sequential colimit of $\mathrm{Fin}(k)^S$. 
  By \Cref{ScottFiniteCodomain}, $\mathrm{Fin}(k)^S$ is the sequential colimit of the $\mathrm{Fin}(k)^{S_n}$ and we can swap the sequential colimits to conclude.
  \end{proof}

\subsection{$\Closed$ and $\Stone$}

\begin{corollary}\label{TruncationStoneClosed}
  For all $S:\Stone$, the proposition $\propTrunc{S}$ is closed. 
\end{corollary}
\begin{proof}
  By \Cref{SpectrumEmptyIff01Equal}, $\neg S$ is equivalent to $0=_{2^S} 1$, which is open by \Cref{BooleIsODisc} and \Cref{OdiscQuotientCountableByOpen}. 
  Hence $\neg \neg S$ is a closed proposition, and by \Cref{LemSurjectionsFormalToCompleteness}, so is $\propTrunc{S}$. 
\end{proof}

\begin{corollary}\label{PropositionsClosedIffStone}
  A proposition $P$ is closed if and only if it is a Stone space. 
\end{corollary}
\begin{proof}
  By the above, if $S$ is both a Stone space and a proposition, it is closed. 
  By \Cref{ClosedPropAsSpectrum}, any closed proposition is Stone. 
\end{proof}

\begin{lemma}\label{StoneEqualityClosed}
For all $S:\Stone$ and $s,t:S$, the proposition $s=t$ is closed. 
\end{lemma}
\begin{proof}
  Suppose $S= Sp(B)$ and let $G$ be a countable set of generators for $B$. 
  Then $s=t$ if and only if $s(g) = t(g)$ for all $g:G$. 
  So $s=t$ is a countable conjunction of decidable propositions, hence 
  closed.
\end{proof}

\subsection{The topology on Stone spaces}
\begin{theorem}\label{StoneClosedSubsets}
  Let $A\subseteq S$ be a subset of a Stone space. The following are equivalent:
  \begin{enumerate}[(i)]
    \item There exists a map $\alpha:S \to 2^\N$ such that 
      $A (x) \leftrightarrow \forall_{n:\N} \alpha_{x,n} = 0$ for any $x:S$. 
    \item There exists a family 
      $(D_n)_{n:\N}$ 
      of decidable subsets of $S$ such that $A = \bigcap_{n:\N} D_n$. 
    \item There exists a Stone space $T$ and some embedding $T\to S$ which image is $A$
    \item There exists a Stone space $T$ and some map $T\to S$ which image is $A$. 
    \item $A$ is closed.
  \end{enumerate}
\end{theorem}
\begin{proof}
\item 
  \begin{itemize}
  \item 
    $(i)\leftrightarrow (ii)$. $D_n$ and $\alpha$ can be defined from each other by 
     $D_n(x) \leftrightarrow (\alpha_{x,n} = 0)$. Then observe that
     \[
     x\in \bigcap_{n:\N} D_n \leftrightarrow 
      \forall_{n:\mathbb N} (\alpha_{x,n} = 0) 
     \]
     
   \item $(ii) \to (iii)$. Let $S=Sp(B)$. 
      By \Cref{AxStoneDuality}, we have $(d_n)_{n:\N}$ in $B$ such that $D_n = \{x:S\ |\ x(d_n) = 0\}$. 
      Let $C = B/(d_n)_{n:\N}$.
      Then $Sp(C) \to S$ is as desired because:
      \[Sp(C) = \{x:S\ |\ \forall_{n:\N} x(d_n) =0\}  = \bigcap_{n:\N} D_n.\]
   \item $(iii) \to (iv)$. Immediate.
   \item $(iv) \to (ii)$. Assume $f:T\to S$ corresponds to $g:B\to C$ in $\Boole$. 
     By \Cref{BooleEpiMono}, $f(T) = Sp(B/Ker(g))$, and 
     there is a surjection $d:\N\to Ker(g)$. Denote by $D_n$ the decidable subset of $S$ corresponding to $d_n$. Then we have that $Sp(B/Ker(g)) = \bigcap_{n:\N} D_n$. 
   \item $(i) \to (v)$. By definition.
   \item $(v) \to (iv)$.
     We have a surjection $2^\N\to\Closed$ defined by $\alpha \mapsto \forall_{n:\mathbb N} \alpha_n = 0.$
     \Cref{LocalChoiceSurjectionForm} 
     gives us that there merely exists $T, e, \beta_\cdot$ as follows:
     \[
       \begin{tikzcd}
         T \arrow[r,"\beta"] \arrow[d, two heads,swap,"e"] & 2^\mathbb N 
         \arrow[d,two heads] \\
         S \arrow[r,swap,"A"] & \Closed
       \end{tikzcd} 
     \] 
     Define $B(x) \leftrightarrow \forall_{n:\mathbb N} \beta_{x,n} = 0$. 
     As $(i) \to (iii)$ by the above, $B$ is the image of some Stone space. 
     Note that $A$ is the image of $B$, 
     thus $A$ is the image of some Stone space. 
     \end{itemize} 
     \end{proof} 

\begin{corollary}\label{ClosedInStoneIsStone}
Closed subtypes of Stone spaces are Stone.
\end{corollary}

\begin{corollary}\label{InhabitedClosedSubSpaceClosed}
  For $S:\Stone$ and $A\subseteq S$ closed, we have 
  $\exists_{x:S} A(x)$ is closed. 
\end{corollary}
\begin{proof}
  By \Cref{ClosedInStoneIsStone}, $\Sigma_{x:S}A(x)$ is Stone, 
  so its truncation is closed by \Cref{TruncationStoneClosed}.
\end{proof}

\begin{corollary}\label{ClosedDependentSums}
  Closed propositions are closed under sigma types. 
\end{corollary}
\begin{proof}
  Let $P:\Closed$ and $Q:P \to \Closed$. 
  Then $\Sigma_{p:P} Q(p) \leftrightarrow \exists_{p:P} Q(p)$.
  As $P$ is Stone by \Cref{PropositionsClosedIffStone}, 
  \Cref{InhabitedClosedSubSpaceClosed} gives that $\Sigma_{p:P} Q(p)$ is closed. 
\end{proof}
\begin{remark}\label{ClosedDominance}\label{ClosedTransitive}
  Analogously to \Cref{OpenTransitive} and \Cref{OpenDominance}, it follows that 
  closedness is transitive and $\Closed$ forms a dominance. 
\end{remark}



\begin{lemma}\label{StoneSeperated}
  Assume $S:\Stone $ with $F,G:S \to \Closed$ be such that $F\cap G = \emptyset$. 
  Then there exists a decidable subset $D:S \to 2$ such $F\subseteq D, G \subseteq \neg D$. 
\end{lemma}
\begin{proof}
  Assume $S = Sp(B)$. 
  By \Cref{StoneClosedSubsets}, for all $n:\N$ there is $f_n,g_n:B$ such that 
  $x\in F$ if and only if $\forall_{n:\N}x(f_n) = 0$ and 
  $y\in G$ if and only if $\forall_{n:\N}y(g_n) = 0$.
%
  Denote by $h$ the sequence define by $h_{2k}=f_k$ and $h_{2k+1}=g_k$.
Then $Sp(B/(h_k)_{k:\N}) = F \cap G=\emptyset$, so by \Cref{SpectrumEmptyIff01Equal}
%
  there exists finite sets $I,J\subseteq \N $ such that 
  $1 =_B ((\bigvee_{i:I}  f_i) \vee (\bigvee_{j:J}  g_j)).$
  If $y\in F$, then $y(f_i) = 0$ for all $i:I$, hence
  $y(\bigvee_{j:J} g_j) = 1 $
 If $x\in G$, we have 
  $x(\bigvee_{j:J} g_j) = 0$. 
  Thus we can define the required $D$ by 
  $D(x) \leftrightarrow x(\bigvee_{j:J} g_j) = 1$.
\end{proof} 


\section{Compact Hausdorff spaces}
\begin{definition}
  A type $X$ is called a compact Hausdorff space if its identity types are closed propositions and there exists some $S:\Stone$ and a surjection $S\twoheadrightarrow X$.
\end{definition}


\subsection{Topology on compact Hausdorff spaces}

\begin{lemma}\label{CompactHausdorffClosed}
  Let $X:\CHaus$ with $S:\Stone$ and a surjective map $q:S\twoheadrightarrow X$.
  Then $A\subseteq X$ is closed if and only if it is the image of a closed subset of $S$ by $q$. 
\end{lemma}
\begin{proof}
  As $q$ is surjective, we have $q(q^{-1}(A)) = A$.
  If $A$ is closed, so is $q^{-1}(A)$ and 
  hence $A$ is the image of a closed subtype of $S$. 
  Conversely, let $B\subseteq S$ be closed. 
  Define $A'\subseteq S$ by 
  \[A'(s) = \exists_{t:S} (B(t) \wedge q(s) = q(t)).\]
  Note that $B(t)$ and $q(s) = q(t)$ are closed. 
  Hence by \Cref{InhabitedClosedSubSpaceClosed}, $A'$ is closed. 
  Also $A'$ factors through $q$ as a map $A: X\to \Closed$.
  Furthermore, $A'(s) \leftrightarrow (q(s)\in q(B))$. 
  Hence $A=q(B)$. 
\end{proof}

\begin{remark}\label{InhabitedClosedSubSpaceClosedCHaus}
  Let $X:\Chaus$.
  From \Cref{StoneClosedSubsets}, it follows that $A\subseteq X$ is closed if and only if it is the image of a map 
  $T\to X$ for some $T:\Stone$. 
  If $A$ is closed, it follows from \Cref{InhabitedClosedSubSpaceClosed} that $\exists_{x:X} A(x)$ is closed as well, 
  hence $\neg\neg$-stable, and equivalent to $A \neq \emptyset$. 
\end{remark}

\begin{corollary}\label{AllOpenSubspaceOpen}
  For $U\subseteq X$ an open subset of a compact Hausdorff space, 
  $\forall_{x:X} U(x)$ is open. 
\end{corollary}

\begin{lemma}\label{CHausFiniteIntersectionProperty}
  Given $X:\Chaus$ and $C_n:X\to \Closed$ closed subsets such that $\bigcap_{n:\N} C_n =\emptyset$, there is some $k:\N$ 
  with $\bigcap_{n\leq k} C_n  = \emptyset$. 
\end{lemma}
\begin{proof}
  By \Cref{CompactHausdorffClosed} it is enough to prove the result when $X$ is Stone, and by \Cref{StoneClosedSubsets} we can assume $C_n$ decidable.
  So assume 
  $X=Sp(B)$ and $c_n:B$ such that: 
  \[C_n = \{x:B\to 2\ |\ x(c_n) = 0\}.\]
  Then the set of maps $B\to 2$ sending all $c_n$ to $0$ is given by: 
  \[Sp(B/(c_n)_{n:\N})
  \simeq \bigcap_{n:\N} C_n = \emptyset .\]
  Hence 
  $0=1$ in $B/(c_n)_{n:\N}$ 
  and there is some $k:\N$ with 
  $\bigvee_{n\leq k} c_n = 1$, which also means that: 
  \[\emptyset = Sp(B/(c_n)_{n\leq k}) 
  \simeq \bigcap_{n\leq k} C_n \]
  as required.
\end{proof}

\begin{corollary}\label{ChausMapsPreserveIntersectionOfClosed}
  Let $X,Y:\CHaus$ and $f:X \to Y$. 
  Suppose $(G_n)_{n:\N}$ is a decreasing sequence of closed subsets of $X$. 
  Then $f(\bigcap_{n:\N} G_n) = \bigcap_{n:\N}f(G_n)$. 
\end{corollary}
\begin{proof}
  It is always the case that $f(\bigcap_{n:\N} G_n) \subseteq \bigcap_{n:\N} f(G_n)$. 
  For the converse direction, suppose that $y \in f(G_n)$ for all $n:\N$. 
  We define $F\subseteq X$ closed by $F=f^{-1}(y)$. 
  Then for all $n:\N$ we have that $F\cap G_n$ is 
  non-empty. 
  By \Cref{CHausFiniteIntersectionProperty} this implies that $\bigcap_{n:\N}(F\cap G_n) \neq \emptyset$. 
  By \Cref{InhabitedClosedSubSpaceClosedCHaus},  
  $\bigcap_{n:\N} (F\cap G_n)$ is merely inhabited. Thus $y\in f(\bigcap_{n:\N} G_n)$ as required. 
\end{proof}

\begin{corollary}\label{CompactHausdorffTopology}
Let $A\subseteq X$ be a subset of a compact Hausdorff space and $p:S\twoheadrightarrow X$ be a surjective map with $S:\Stone$. Then $A$ is closed (resp. open) if and only if there exists a sequence $(D_n)_{n:\N}$ of decidable subsets of $S$ such that $A = \bigcap_{n:\N} p(D_n)$ (resp. $A = \bigcup_{n:\N} \neg p(D_n)$).
\end{corollary}
\begin{proof}
  The characterization of closed sets follows from characterization (ii) in \Cref{StoneClosedSubsets}, 
  \Cref{CompactHausdorffClosed} 
  and \Cref{ChausMapsPreserveIntersectionOfClosed}. 
  For open sets we use \Cref{rmkOpenClosedNegation} and
  \Cref{ClosedMarkov}.
\end{proof}
\begin{remark}
  For $S:\Stone$, there is a surjection $\N\twoheadrightarrow 2^S$. 
  It follows that for any $X:\CHaus$ there is a surjection from $\N$ to a basis of $X$. 
  Classically this means that $X$ is second countable. 
\end{remark}
%
\begin{lemma}\label{CHausSeperationOfClosedByOpens}
 Assume $X:\CHaus$ and $A,B\subseteq X$ closed such that $A\cap B=\emptyset$. 
  Then there exist $U,V\subseteq X$ open such that $A\subseteq U$, $B\subseteq V$ and $U\cap V=\emptyset$. 
\end{lemma}
\begin{proof}
  Let $q:S\to X$ be a surjective map with $S:\Stone$.
  As $q^{-1}(A)$ and $q^{-1}(B)$ are closed, 
  by \Cref{StoneSeperated}, there is some $D:S \to 2$ such that
  $q^{-1}(A) \subseteq D$ and $q^{-1}(B) \subseteq \neg D$. 
  Note that $q(D)$ and $q(\neg D)$ are closed by \Cref{CompactHausdorffClosed}. 
  We define $U = \neg q(\neg D) \cap \neg B$ and $V=\neg  q(D) \cap \neg A$. As $q^{-1}(A) \cap \neg D  =\emptyset$, we have that 
  $A\subseteq \neg q(\neg D)$. As $A\cap B = \emptyset$, we have that $A\subseteq \neg B$ so $A\subseteq U$. Similarly $B\subseteq V$. 
  Then $U$ and $V$ are disjoint because $\neg q(D)\cap \neg q(\neg D) \subseteq \neg (q(D)\cup q(\neg D)) = \neg X = \emptyset$.
\end{proof}

\subsection{Compact Hausdorff spaces are stable under sigma types}

\begin{lemma}\label{StoneAsClosedSubsetOfCantor}
A type $X$ is Stone if and only if it is merely a closed subset of $2^\N$.
\end{lemma}
\begin{proof}
  By \Cref{BooleAsCQuotient}, any $B:\Boole$ can be written as $2[\N]/(r_n)_{n:\N}$. 
  By \Cref{BooleEpiMono}, the quotient map induces an embedding $Sp(B)\hookrightarrow Sp(2[\N])= 2^\N$, 
  which is closed by \Cref{StoneClosedSubsets}.
\end{proof}


\begin{lemma}\label{SigmaCompactHausdorff}
Compact Hausdorff spaces are stable under $\Sigma$-types.
\end{lemma}

\begin{proof}
Assume $X:\CHaus$ and $Y:X\to \CHaus$. By \Cref{ClosedDependentSums} we have that identity type in $\Sigma_{x:X}Y(x)$ are closed. By \Cref{StoneAsClosedSubsetOfCantor} we know that for any $x:X$ there merely exists a closed $C\subseteq 2^\N$ with a surjection $\Sigma_{2^\N}C \twoheadrightarrow Y(x)$. By local choice we merely get $S:\Stone$ with a surjection $p:S\to X$ such that for all $s:S$ we have $C_s\subseteq 2^\N$ closed and a surjection $\Sigma_{2^\N}C_s\twoheadrightarrow Y(p(s))$. This gives a surjection $\Sigma_{s:S,x:2^\N}C_s(x)\twoheadrightarrow\Sigma_{x:X}Y_x$ and the source is Stone by \Cref{StoneClosedUnderPullback} and \Cref{ClosedInStoneIsStone}.
\end{proof}

\subsection{Stone spaces are stable under sigma types}
We will show that Stone spaces are precisely totally disconnected compact Hausdorff spaces. 
We will use this to prove that a sigma type of Stone spaces is Stone.

\begin{lemma}\label{AlgebraCompactHausdorffCountablyPresented}
Assume $X:\Chaus$, then $2^X$ is countably presented.
\end{lemma}

\begin{proof}
  There is some surjection $q:S\twoheadrightarrow X$ with $S:\Stone$. 
%
  This induces an injection of Boolean algebras $2^X \hookrightarrow 2^S$.
  Note that $a:S\to 2$ lies in $2^X$ if and only if: 
  \[\forall_{s,t:S}\ q(s) =_X q(t) \to a(s)=a(t).\]
  As equality in $X$ is closed and equality in $2$ is decidable, the implication is open for every $s,t:S$. 
  By \Cref{AllOpenSubspaceOpen}, we conclude that
  $2^X$ is an open subalgebra of $2^S$. 
  Therefore, it is in $\ODisc$ by
  \Cref{PropOpenIffOdisc} and \Cref{OdiscSigma} 
  and in $\Boole$ by \Cref{ODiscBAareBoole}.
%
%
%
\end{proof}
\begin{definition}
For all $X:\Chaus$ and $x:X$,
  we define $Q_x$ the connected component of $x$
  as the intersection of all $D\subseteq X$ decidable such that $x\in D$. 
\end{definition}

\begin{lemma}\label{ConnectedComponentClosedInCompactHausdorff}
For all $X:\CHaus$ with $x:X$, we have that $Q_x$ is a countable intersection of decidable subsets of $X$.
\end{lemma}
\begin{proof}
  By \Cref{AlgebraCompactHausdorffCountablyPresented},
  we can enumerate the elements of $2^X$, say as $(D_n)_{n:\N}$. 
  For $n:\N$ we define $E_n$ as $D_n$ if $x\in D_n$ and $X$ otherwise. 
  Then $\cap_{n:\N}E_n = Q_x$.
\end{proof}

\begin{lemma}\label{ConnectedComponentSubOpenHasDecidableInbetween}
  Assume $X:\Chaus$ with $x:X$ and suppose $U\subseteq X$ open with $Q_x\subseteq U$. 
  Then we have some decidable $E\subseteq X$ with $x\in E$ and $E\subseteq U$. 
\end{lemma}
\begin{proof}
  By \Cref{ConnectedComponentClosedInCompactHausdorff}, 
  we have $Q_x = \bigcap_{n:\N}D_n$ with $D_n\subseteq X$ decidable. 
  If $Q_x \subseteq U$, then
  \[Q_x\cap \neg U = \bigcap_{n:\N} (D_n \cap \neg U) = \emptyset.\]
  By \Cref{CHausFiniteIntersectionProperty} there is some $k:\N$ with 
  \[(\bigcap_{n\leq k} D_n )\cap \neg U  = \bigcap_{n\leq k} (D_n \cap \neg U) = \emptyset.\]
  Therefore $\bigcap_{n\leq k} D_n \subseteq \neg\neg U$.
  As $U$ is open, $\neg \neg U = U$ and 
  $E:= \bigcap_{n\leq k} D_n$ is as desired. 
\end{proof}

\begin{lemma}\label{ConnectedComponentConnected}
Assume $X:\Chaus$ with $x:X$. Then any map in $Q_x\to 2$ is constant.
\end{lemma}
\begin{proof}
Assume $Q_x = A\cup B$ with $A,B$ decidable and disjoint subsets of $Q_x$. Assume $x\in A$. 
By \Cref{ConnectedComponentClosedInCompactHausdorff}, $Q_x\subseteq X$ is closed. 
Using \Cref{ClosedTransitive}, it follows that $A,B\subseteq X$ are closed and disjoint.
By \Cref{CHausSeperationOfClosedByOpens} there exist $U,V$ disjoint open such that $A\subseteq U$ and $B\subseteq V$. 
By \Cref{ConnectedComponentSubOpenHasDecidableInbetween} we have a decidable $D$ such that $Q_x\subseteq D\subseteq U\cup V$. 
Note that $E := D\cap U = D \cap (\neg V)$ is clopen, hence decidable by \Cref{ClopenDecidable}.
But $x\in E$, hence $B\subseteq Q_x \subseteq E$ but $B \cap E = \emptyset$, hence $B=\emptyset$. 
\end{proof}

\begin{lemma}\label{StoneCompactHausdorffTotallyDisconnected}
Let $X:\CHaus$, then $X$ is Stone if and only $\forall_{x:X}\ Q_x=\{x\}$.
\end{lemma}

\begin{proof}
  By \Cref{AxStoneDuality}, it is clear that for all $x:S$ with $S:\Stone$ we have that $Q_x=\{x\}$.
  Conversely, assume $X:\CHaus$ such that $\forall_{x:X}\ Q_x = \{x\}$.
  We claim that the evaluation map $e:X \to Sp(2^X)$ is both injective and surjective, hence an equivalence. 
    Let $x,y:X$ be such that $e(x)=e(y)$, i.e. such that $f(x) = f(y)$ for all $f:2^X$. Then $y \in Q_x$, hence $x=y$ by assumption. Thus $e$ is injective. 
    Let $q:S\twoheadrightarrow X$ be a surjective map. 
    It induces an injection $2^X \hookrightarrow 2^S$, which by \Cref{SurjectionsAreFormalSurjections}
    induces a surjection $p:Sp(2^S) \twoheadrightarrow Sp(2^X)$. 
    Note that $e\circ q$ is equal to $p$ so $e$ is surjective. 
%
%
%
\end{proof}

\begin{theorem}
  \label{stone-sigma-closed}
Assume $S:\Stone$ and $T:S\to\Stone$. Then $\Sigma_{x:S}T(x)$ is Stone.
\end{theorem}

\begin{proof}
By \Cref{SigmaCompactHausdorff} we have that $\Sigma_{x:S}T(x)$ is compact Hausdorff. 
By \Cref{StoneCompactHausdorffTotallyDisconnected} 
it is enough to show that for all $x:S$ and $y:T(x)$ 
we have that $Q_{(x,y)}$ is a singleton.
Assume $(x',y')\in Q_{(x,y)}$, then for any map $f:S\to 2$ we have that:
\[ f(x) = f\circ \pi_1(x,y) = f\circ \pi_1(x',y') = f(x')\]
so that $x'\in Q_x$ and since $S$ is Stone, by \Cref{StoneCompactHausdorffTotallyDisconnected} we have that $x=x'$.
Therefore we have $Q_{(x,y)}\subseteq \{x\}\times T(x)$. 
Assume $z,z':Q_{(x,y)}$, then for any map $g:T(x)\to 2$ we have that $g(z)=g(z')$ by 
\Cref{ConnectedComponentConnected}. Since $T(x)$ is Stone, 
we conclude $z=z'$ by \Cref{StoneCompactHausdorffTotallyDisconnected}.
\end{proof}

\section{The unit interval as a Compact Hausdorff space}
Since we have dependent choice, the unit interval $\mathbb I = [0,1]$ can be defined using 
Cauchy reals or Dedekind reals. 
We can freely use results from constructive analysis \cite{Bishop}. 
As we have $\neg$WLPO, MP and LLPO, we can use the results from 
constructive reverse mathematics that follow from these principles \cite{ReverseMathsBishop, HannesDiener}. 
\begin{definition}
  \label{def-cs-Interval}
  We define for each $n:\N$ the Stone space $2^n$ of binary sequences of length $n$.
  And we define $cs_n:2^n \to \mathbb Q$ by 
  $cs_n(\alpha) = \sum_{i < n } \frac{\alpha(i)}{2^{i+1}}.$
  Finally we write $\sim_n$ for the binary relation on $2^n$ given by 
  $\alpha\sim_n \beta 
  \leftrightarrow \left|cs_n(\alpha) - cs_n(\beta)\right|\leq\frac{1}{2^n}$.
\end{definition}
\begin{remark}
  The inclusion $Fin(n) \hookrightarrow \N$ induces a restriction 
  $\_|_n : 2^\N \to 2^n$ for each $n:\N$. 
\end{remark}
\begin{definition}
  We define $cs:2^\N \to \I$ as 
  $cs(\alpha) = 
  \sum_{i :\N } \frac{\alpha(i)}{2^{i+1}}.
  $
\end{definition}

\begin{theorem}\label{IntervalIsCHaus}
  $\I$ is compact Hausdorff.
\end{theorem}
\begin{proof}
  By LLPO, we have that $cs$ is surjective.   
  Note that $cs(\alpha) = cs(\beta)$ if and only if 
  for all $n:\N$ we have $\alpha|_n \sim_n \beta|_n$. 
  This is a countable conjunction of decidable propositions.
\end{proof}

\begin{remark}
  Following Definitions 2.7 and 2.10 of \cite{Bishop}, we have that $x<y$ is open for all $x,y:\I$. Hence open intervals are open. 
\end{remark}
\begin{lemma}\label{ImageDecidableClosedInterval}
  For $D\subseteq 2^\N$ decidable, we have $cs(D)$ a finite union of closed intervals. 
\end{lemma}
\begin{proof}
  If $D$ is given by those $\alpha:2^\N$ with a fixed initial segment, $cs(D)$ is a closed interval. 
  Any decidable subset of $2^\N$ is a finite union of such subsets. 
\end{proof}
\begin{lemma}\label{complementClosedIntervalOpenIntervals}
  The complement of a finite union of closed intervals is 
  a finite union of open intervals. 
\end{lemma}
By \Cref{CompactHausdorffTopology} we can thus conclude:
\begin{lemma}\label{IntervalTopologyStandard}
  Every open $U\subseteq \I$ can be written as a countable union of open intervals.
\end{lemma} 
%
  It follows that the topology of $\I$ is generated by open intervals, 
  which corresponds to the standard topology on $\I$. 
  Hence our notion of continuity agrees with the 
  $\epsilon,\delta$-definition of continuity one would expect and we get the following:
\begin{theorem}
  Every function $f:\I\to \I$ is continuous in the $\epsilon,\delta$-sense. 
\end{theorem}

\section{Cohomology}
In this section we compute $H^1(S,\Z) = 0$ for $S$ Stone, and show that $H^1(X,\Z)$ for $X$ compact Hausdorff can be computed using \v{C}ech cohomology. We then apply this to compute $H^1(\I,\Z)=0$. 

\begin{remark}
We only work with the first cohomology group with coefficients in $\Z$ as it is sufficient for the proof of Brouwer's fixed-point theorem, but the results could be extended to $H^n(X,A)$ for $A$ any family of countably presented abelian groups indexed by $X$.
\end{remark}

\begin{remark}
We write $Ab$ for the type of Abelian groups and if $G:Ab$ we write $\B G$ for the delooping of $G$ \cite{hott,davidw23}. This means that $H^1(X,G)$ is the set truncation of $X \to \B G$. 
\end{remark}

\subsection{\v{C}ech cohomology}

\begin{definition}
Given a type $S$, types $T_x$ for $x:S$ and $A:S\to\mathrm{Ab}$, we define $\check{C}(S,T,A)$ as the chain complex:
\[
\begin{tikzcd}
     \prod_{x:S}A_x^{T_x} \ar[r,"d_0"] & \prod_{x:S}A_x^{T_x^2}\ar[r,"d_1"] &  \prod_{x:S}A_x^{T_x^3}
\end{tikzcd}
\]
with the usual boundary maps:
\begin{align*}
d_0(\alpha)_x(u,v) =&\ \alpha_x(v)-\alpha_x(u)\\
d_1(\beta)_x(u,v,w) =&\ \beta_x(v,w) - \beta_x(u,w) + \beta_x(u,w)
\end{align*}
\end{definition}

\begin{definition}
Given a type $S$, types $T_x$ for $x:S$ and $A:S\to\mathrm{Ab}$, we define its \v{C}ech cohomology groups by:
\[
  \check{H}^0(S,T,A) = \mathrm{ker}(d_0)\quad \quad \quad \check{H}^1(S,T,A) = \mathrm{ker}(d_1)/\mathrm{im}(d_0)
\]
We call elements of $\mathrm{ker}(d_1)$ cocycles and elements of $\mathrm{im}(d_0)$ coboundaries.
\end{definition}

This means that $\check{H}^1(S,T,A) = 0$ if and only if $\check{C}(S,T,A)$ is exact. Now we give three general lemmas about \v{C}ech complexes.

\begin{lemma}\label{section-exact-cech-complex}
Assume a type $S$, types $T_x$ for $x:S$ and $A:S\to\mathrm{Ab}$ with $t:\prod_{x:S}T_x$. Then $\check{C}(S,T,A)$ is exact.
\end{lemma}

\begin{proof}
Assume given a cocycle, i.e. $\beta:\prod_{x:S}A_x^{T_x^2}$ such that for all $x:S$ and $u,v,w:T_x$ we have that $\beta_x(u,v)+\beta_x(v,w) = \beta_x(u,w)$. We define $\alpha:\prod_{x:S}A_x^{T_x}$ by $\alpha_x(u) = \beta_x(t_x,u)$. Then for all $x:S$ and $u,v:T_x$ we have that $d_0(\alpha)_x(u,v) =  \beta_x(t_x,v) - \beta_x(t_x,u) = \beta_x(u,v)$ so that $\beta$ is a coboundary.
\end{proof}

\begin{lemma}\label{canonical-exact-cech-complex}
Given a type $S$, types $T_x$ for $x:S$ and $A:S\to\mathrm{Ab}$, we have that $\check{C}(S,T,\lambda x.A_x^{T_x})$ is exact.
\end{lemma}

\begin{proof}
Assume given a cocycle, i.e. $\beta:\prod_{x:S}A_x^{T_x^3}$ such that for all $x:S$ and $u,v,w,t:T_x$ we have that $\beta_x(u,v,t)+\beta_x(v,w,t) = \beta_x(u,w,t)$. We define $\alpha:\prod_{x:S}A_x^{T_x^2}$ by $\alpha_x(u,t) = \beta_x(t,u,t)$. Then for all $x:S$ and $u,v,t:T_x$ we have that $d_0(\alpha)_x(u,v,t) = \beta_x(t,v,t) - \beta_x(t,u,t) = \beta_x(u,v,t)$ so that $\beta$ is a coboundary.
\end{proof}

\begin{lemma}\label{exact-cech-complex-vanishing-cohomology}
Assume a type $S$ and types $T_x$ for $x:S$ such that $\prod_{x:S}\propTrunc{T_x}$ and $A:S\to\mathrm{Ab}$ such that $\check{C}(S,T,A)$ is exact.
Then given $\alpha:\prod_{x:S}\B A_x$ with $\beta:\prod_{x:S} (\alpha(x) = *)^{T_x}$, we can conclude $\alpha = *$.
\end{lemma}

\begin{proof}
We define $g : \prod_{x:S} A_x^{T_x^2}$ by $g_x(u,v) = \beta_x(u)^{-1}\cdot\beta_x(v)$.
It is a cocycle in the \v{C}ech complex, so that by exactness there is $f:\prod_{x:S}A_x^{T_x}$ such that for all $x:S$ and $u,v:T_x$ we have that $g_x(u,v)=f_x(u)^{-1}\cdot f_x(v)$.
Then we define $\beta' : \prod_{x:S}(\alpha(x)=*)^{T_x}$ by $\beta'_x(u) = \beta_x(u)\cdot f_x(u)^{-1}$
so that for all $x:S$ and $u,v:T_x$ we have that $\beta'_x(u) = \beta'_x(v)$ is equivalent to $f_x(u)^{-1}\cdot f_x(v) = \beta_x(u)^{-1}\cdot\beta_x(v)$, which holds by definition. Therefore $\beta'$ factors through $S$, giving a proof of $\alpha = *$.
\end{proof}

\subsection{Cohomology of Stone spaces}

%
%
%
\begin{lemma}\label{finite-approximation-surjection-stone}
Assume given $S:\Stone$ and $T:S\to\Stone$ such that $\prod_{x:S}\propTrunc{T(x)}$.
Then there exists a sequence of finite types $(S_k)_{k:\N}$ with limit $S$ 
and a compatible sequence of families of finite types $T_k$ over $S_k$
with $\prod_{x:S_k}\propTrunc{T_k(x)}$ and 
$\mathrm{lim}_k\left(\sum_{x:S_k}T_k(x)\right) = \sum_{x:S}T(x)$. 
%
\end{lemma}

\begin{proof}
  This follows from \Cref{ProFiniteMapsFactorization} and \Cref{stone-sigma-closed}.
\end{proof}

\begin{lemma}\label{cech-complex-vanishing-stone}
Assume given $S:\Stone$ with $T:S\to\Stone$ such that $\prod_{x:S}\propTrunc{T_x}$. Then we have that $\check{C}(S,T,\Z)$ is exact.
\end{lemma}

\begin{proof}
We apply \cref{finite-approximation-surjection-stone} to get $S_k$ and $T_k$ finite. Then by \cref{scott-continuity} we have that $\check{C}(S,T,\Z)$ is the sequential colimit of the $\check{C}(S_k,T_k,\Z)$. By \cref{section-exact-cech-complex} we have that each of the $\check{C}(S_k,T_k,\Z)$ is exact, and a sequential colimit of exact sequences is exact.
\end{proof}

\begin{lemma}\label{eilenberg-stone-vanish}
Given $S:\Stone$, we have that $H^1(S,\Z) = 0$. 
\end{lemma}

\begin{proof}
Assume given a map $\alpha:S\to \B\Z$. We use local choice to get $T:S\to\Stone$ such that $\prod_{x:S}\propTrunc{T_x}$ with $\beta:\prod_{x:S}(\alpha(x)=*)^{T_x}$. Then we conclude by \cref{cech-complex-vanishing-stone} and \cref{exact-cech-complex-vanishing-cohomology}.
\end{proof}

\begin{corollary}\label{stone-commute-delooping}
For any $S:\Stone$ the canonical map $\B(\Z^S) \to (\B\Z)^S$ is an equivalence.
\end{corollary}

\subsection{\v{C}ech cohomology of compact Hausdorff spaces}

\begin{definition}
A \v{C}ech cover consists of $X:\CHaus$ and $S:X\to\Stone$ such that $\prod_{x:X}\propTrunc{S_x}$ and $\sum_{x:X}S_x:\Stone$.
\end{definition}

By definition any compact Hausdorff type has a \v{C}ech cover.

\begin{lemma}\label{cech-eilenberg-0-agree}
Given a \v{C}ech cover $(X,S)$, we have that $H^0(X,\Z) = \check{H}^0(X,S,\Z)$.
\end{lemma}

\begin{proof}
By definition an element in $\check{H}^0(X,S,\Z)$ is a map $f:\prod_{x:X}\Z^{S_x}$
such that for all $u,v:S_x$ we have $f(u)=f(v)$. Since $\Z$ is a set and the $S_x$ are merely inhabited, this is equivalent to $\Z^X$.
\end{proof}

\begin{lemma}\label{eilenberg-exact}
Given a \v{C}ech cover $(X,S)$ we have an exact sequence:
\[H^0(X,\lambda x.\Z^{S_x}) \to H^0(X,\lambda x.\Z^{S_x}/\Z) \to H^1(X,\Z)\to 0\]
\end{lemma}

\begin{proof}
We use the long exact cohomology sequence associated to:
\[0 \to \Z \to \Z^{S_x} \to \Z^{S_x}/\Z\to 0\]
We just need $H^1(X,\lambda x.\Z^{S_x}) = 0$ to conclude. But by \cref{stone-commute-delooping} we have that $H^1(X,\lambda x.\Z^{S_x}) = H^1\left(\sum_{x:X}S_x,\Z\right)$ which vanishes by \cref{eilenberg-stone-vanish}.
\end{proof}

\begin{lemma}\label{cech-exact}
Given a \v{C}ech cover $(X,S)$ we have an exact sequence:
\[\check{H}^0(X,\lambda x.\Z^{S_x}) \to \check{H}^0(X,\lambda x.\Z^{S_x}/\Z) \to \check{H}^1(X,\Z)\to 0\]
\end{lemma}

\begin{proof}
By \cref{eilenberg-stone-vanish} and the long exact sequence for cohomology, we have an exact sequence of complexes:
\[0 \to \check{C}(X,S,\Z) \to \check{C}(X,S,\lambda x.\Z^{S_x}) \to \check{C}(X,S,\lambda x.\Z^{S_x}/\Z)\to 0\]
But since $\check{H}^1(X,\lambda x.\Z^{S_x}) = 0$ by \cref{canonical-exact-cech-complex}, we conclude using the associated long exact sequence.
\end{proof}

\begin{theorem}\label{cech-eilenberg-1-agree}
Given a \v{C}ech cover $(X,S)$, we have that $H^1(X,\Z) = \check{H}^1(X,S,\Z)$
\end{theorem}

\begin{proof}
We apply \cref{cech-eilenberg-0-agree}, \cref{eilenberg-exact} and \cref{cech-exact}.
\end{proof}

This means that \v{C}ech cohomology does not depend on $S$.

\subsection{Cohomology of the interval}
%
%
\begin{remark}\label{description-Cn-simn}
  Recall from \Cref{def-cs-Interval} that 
  there is a binary relation $\sim_n$ on $2^n=:\I_n$ such that 
  $(2^n,\sim_n)$ is equivalent to  $(\mathrm{Fin}(2^n),\lambda x,y.\ |x-y|\leq 1)$
  and for $\alpha,\beta:2^\N$ we have $(cs(\alpha) = cs(\beta)) \leftrightarrow 
  \left(\forall_{n:\N}\alpha|_n \sim_n \beta|_n\right)$. 
\end{remark}

We define $\I_n^{\sim2} = \Sigma_{x,y:\I_n}x\sim_n y$ and $\I_n^{\sim3} = \Sigma_{x,y,z:\I_n}x\sim_n y \land y\sim_n z\land x\sim_n z$.

\begin{lemma}\label{Cn-exact-sequence}
For any $n:\N$ we have an exact sequence:
\[0\to \Z\to \Z^{\I_n} \to \Z^{\I_n^{\sim2}} \to \Z^{\I_n^{\sim3}}\]
with the obvious boundary maps.
\end{lemma}

\begin{proof}
It is clear that the map $\Z\to \Z^{\I_n}$ is injective as $\I_n$ is inhabited, so the sequence is exact at $\Z$. Assume a cocycle $\alpha:\Z^{\I_n}$, meaning that for all $u,v:\I_n$, if $u\sim_nv$ then $\alpha(u)=\alpha(v)$. Then by \cref{description-Cn-simn} we see that $\alpha$ is constant, so the sequence is exact at $\Z^{\I_n}$.

Assume a cocycle $\beta:\Z^{\I_n^{\sim2}}$, meaning that for all $u,v,w:\I_n$ such that $u\sim_nv$, $v\sim_nw$ and $u\sim_nw$ we have that $\beta(u,v)+\beta(v,w) = \beta(u,w)$. 
Using \cref{description-Cn-simn} we can define $\alpha(n) = \beta(0,1)+\cdots+\beta(n-1,n)$.
We can check that $\beta(m,n) = \alpha(n)-\alpha(m)$, so $\beta$ is indeed a coboundary and the sequence is exact at $\Z^{\I_n^{\sim2}}$.
\end{proof}

\begin{proposition}\label{cohomology-I}
We have that $H^0(\I,\Z) = \Z$ and $H^1(\I,\Z) = 0$.
\end{proposition}

\begin{proof}
Consider $cs:2^\N\to\I$ and the associated \v{C}ech cover $T$ of $\I$ defined by: 
\[T_x = \Sigma_{y:2^\N} (x=_\I cs(y))\]
Then for $l=2,3$ we have that $\mathrm{lim}_n\I_n^{\sim l} = \sum_{x:\I} T_x^l$. By \cref{Cn-exact-sequence} and stability of exactness under sequential colimit, we have an exact sequence:
\[ 0\to \Z\to \mathrm{colim}_n \Z^{\I_n} \to \mathrm{colim}_n \Z^{\I_n^{\sim2}}\to \mathrm{colim}_n \Z^{\I_n^{\sim3}}\]
By \cref{scott-continuity} this sequence is equivalent to:
\[ 0\to \Z\to \Pi_{x:\I}\Z^{T_x} \to  \Pi_{x:\mathbb{I}}\Z^{T_x^2} \to  \Pi_{x:\mathbb{I}}\Z^{T_x^3}\]
So it being exact implies that $\check{H}^0(\I,T,\Z) = \Z$ and $\check{H}^1(\I,T,\Z) = 0$.
We conclude by \cref{cech-eilenberg-0-agree} and \cref{cech-eilenberg-1-agree}.
\end{proof}

\begin{remark}
We could carry a similar computation for $\mathbb{S}^1$, by approximating it with $2^n$ with $0^n\sim_n1^n$ added. We would find $H^1(\mathbb{S}^1,\Z)=\Z$.
\end{remark}

\subsection{Brouwer's fixed-point theorem}

Here we consider the modality defined by localising at $\I$ \cite{modalities}, denoted by $L_\I$. We say that $X$ is $\I$-local if $L_\I(X) = X$ and that it is $\I$-contractible if $L_\I(X)=1$.

\begin{lemma}\label{Z-I-local}
$\Z$ and $2$ are $\I$-local.
\end{lemma}

\begin{proof}
By \cref{cohomology-I}, from $H^0(\I,\Z)=\Z$ we get that the map $\Z\to \Z^\I$ is an equivalence, so $\Z$ is $\I$-local. We see that $2$ is $\I$-local as it is a retract of $\Z$.
\end{proof}

\begin{remark}
Since $2$ is $\I$-local, we know by duality that any Stone space is $\I$-local.
\end{remark}

\begin{lemma}\label{BZ-I-local}
$\B\Z$ is $\I$-local.
\end{lemma}

\begin{proof}
Any identity type in $\B\Z$ is a $\Z$-torsor, so it is $\I$-local by \cref{Z-I-local}. So there is at most one factorisation of any map $\I\to \B\Z$ through $1$. From $H^1(\I,\Z)=0$ we get that there merely exists such a factorisation.
\end{proof}

\begin{lemma}\label{continuously-path-connected-contractible}
Assume $X$ a type with $x:X$ such that for all $y:X$ we have $f:\I\to X$ such that $f(0)=x$ and $f(1)=y$. Then $X$ is $\I$-contractible.
\end{lemma}

\begin{proof}
For all $y:X$ we get a map $g:\I\to X\to L_\I(X)$ such that $g(0) = [x]$ and $g(1)=[y]$. Since $L_\I(X)$ is $\I$-local this means that $\prod_{x:X}[*]=[x]$. We conclude $\prod_{y:L_\I(X)}[x]=y$ by applying the elimination principle for the modality.
\end{proof}

\begin{corollary}\label{R-I-contractible}
We have that $\R$ and $\mathbb{D}^2=\{x,y:\mathbb R\ \vert\ x^2+y^2\leq 1\}$ are $\I$-contractible.
\end{corollary}

\begin{proposition}\label{shape-S1-is-BZ}
$L_\I(\R/\Z) = \B\Z$.
\end{proposition}

\begin{proof}
As for any group quotient, the fibers of the map $\R\to\R/\Z$ are $\Z$-torsors, se we have an induced pullback square:
\[
\begin{tikzcd}
\R\ar[r]\ar[d] & 1\ar[d] \\
\R/\Z\ar[r] & \B\Z
\end{tikzcd}
\]
Now we check that the bottom map is an $\I$-localisation. Since $\B\Z$ is $\I$-local by \cref{BZ-I-local}, it is enough to check that its fibers are $\I$-contractible. Since $\B\Z$ is connected it is enough to check that $\R$ is $\I$-contractible, but this is \cref{R-I-contractible}.
\end{proof}

\begin{remark}
By \cref{BZ-I-local}, for any $X$ we have that $H^1(X,\Z) = H^1(L_{\I}(X),\Z)$, so that by \cref{shape-S1-is-BZ} we have that $H^1(\R/\Z,\Z) = H^1(\B\Z,\Z) = \Z$.
\end{remark}

We omit the proof that $\mathbb{S}^1=\{x,y:\R\ \vert\ x^2+y^2=1\}$ is equivalent to $\R/\Z$.
The equivalence can be constructed using trigonometric functions, which exists by Proposition 4.12 in \cite{Bishop}.

\begin{proposition}
\label{no-retraction}
The map $\mathbb{S}^1\to \mathbb{D}^2$ has no retraction.
\end{proposition}

\begin{proof}
By \cref{R-I-contractible} and \cref{shape-S1-is-BZ} we would get a retraction of $\B\Z\to 1$, so $\B\Z$ would be contractible.
\end{proof}

\begin{theorem}[Intermediate value theorem]
  \label{ivt}
  For any $f: \I\to \I$ and $y:\I$ such that $f(0)\leq y$ and $y\leq f(1)$,
  there exists $x:\I$ such that $f(x)=y$.
\end{theorem}

\begin{proof}
  By \Cref{InhabitedClosedSubSpaceClosedCHaus}, the proposition $\exists x:\I.f(x)=y$ is closed and therefore $\neg\neg$-stable, so we can proceed with a proof by contradiction.
  If there is no such $x:\I$, we have $f(x)\neq y$ for all $x:\I$.
  It is a standard fact of constructive analysis \cite{Bishop}, that for different numbers $a,b:\I$, we have $a<b$ or $b<a$, so the following two sets cover $\I$:
  \[
    U_0:= \{x:\I\mid f(x)<y\} \quad\quad
    U_1:= \{x:\I\mid y<f(x)\}
    \]
  Since $U_0$ and $U_1$ are disjoint, we have $\I=U_0+U_1$ which allows us to define a non-constant function $\I\to 2$, which contradicts \Cref{Z-I-local}.
\end{proof}

\begin{theorem}[Brouwer's fixed-point theorem]
  For all $f:\mathbb{D}^2\to \mathbb{D}^2$ there exists $x:\mathbb{D}^2$ such that $f(x)=x$.
\end{theorem}

\begin{proof}
  As above, by \Cref{InhabitedClosedSubSpaceClosedCHaus}, we can proceed with a proof by contradiction,
  so we assume $f(x)\neq x$ for all $x:\mathbb{D}^2$.
  For any $x:\mathbb{D}^2$ we set $d_x= x-f(x)$, so we have that one of the coordinates of $d_x$ is invertible.
  Let $H_x(t) = f(x) + t\cdot d_x $ be the line through $x$ and $f(x)$.
  The intersections of $H_x$ and $\partial\mathbb{D}^2=\mathbb{S}^1$ are given by the solutions of an equation quadratic in $t$. By invertibility of one of the coordinates of $d_x$, there is exactly one solution with $t> 0$.
  We denote this intersection by $r(x)$ and the resulting map $r:\mathbb D^2\to\mathbb S^1$ has the property that it preserves $\mathbb{S}^1$.
  Then $r$ is a retraction from $\mathbb{D}^2$ onto its boundary $\mathbb{S}^1$, which is a contradiction by \Cref{no-retraction}.
\end{proof}

\printbibliography

\end{document}